\documentclass[a4paper,oneside,final]{amsart}
\usepackage[utf8]{inputenc}
\usepackage[a4paper,margin=3cm]{geometry} 
\usepackage{amsmath}
\usepackage{amsthm}
\usepackage{amscd}
\usepackage{amssymb}
\usepackage{MnSymbol}
\usepackage{latexsym}
\usepackage{eucal}
\usepackage{dsfont}
\usepackage{mathtools}
\usepackage{enumitem}
\usepackage[notref,notcite]{showkeys}
\usepackage{todonotes}

\usepackage{verbatim}
\usepackage{graphicx}
\usepackage{float}
\usepackage{array,booktabs}

\usepackage{pdflscape}

\usepackage[colorlinks,pdftex]{hyperref}
\hypersetup{
linkcolor=black,
citecolor=black,
pdftitle={},
pdfauthor={Franziska K\"uhn},
pdfkeywords={},
}

\makeatletter
\renewcommand\section{\@startsection{section}{1}{0mm}{-1.5\baselineskip}{\baselineskip}{\normalsize\bfseries\sffamily}}
\renewcommand\subsection{\@startsection{subsection}{1}{0mm}{-\baselineskip}{\baselineskip}{\normalsize\bfseries\sffamily}}
\makeatother

\makeatletter
\def\@fnsymbol#1{\ensuremath{\ifcase#1\or *\or **\or \dagger\or \ddagger\or
   \mathsection\or \mathparagraph\or \|\or \dagger\dagger
   \or \ddagger\ddagger \else\@ctrerr\fi}}

\newlength{\preskip}
\newlength{\postskip}
\makeatother

\newtheoremstyle{theorem}{\preskip}{\postskip}{\itshape}{}{\bfseries}{}
{.5em}{\textbf{\thmname{#1}\thmnumber{ #2} (\thmnote{ #3})}}
\newtheoremstyle{definition}{\preskip}{\postskip}{\normalfont}{0pt}{\bfseries}{}{.5em}{}
\newtheoremstyle{remark}{\preskip}{\postskip}{\normalfont}{0pt}{\bfseries}{}{.5em}{}

\swapnumbers
\theoremstyle{theorem} \newtheorem{thm}{Theorem}[section]
\theoremstyle{theorem} \newtheorem{lem}[thm]{Lemma}
\theoremstyle{theorem} \newtheorem{prop}[thm]{Proposition}
\theoremstyle{theorem} \newtheorem{kor}[thm]{Corollary}
\theoremstyle{definition} \newtheorem{defn}[thm]{Definition}
\theoremstyle{remark} 
\theoremstyle{remark} 
\theoremstyle{definition} 
\theoremstyle{definition} 
\theoremstyle{remark} \newtheorem{bem}[thm]{Remark}
\theoremstyle{remark} 
\theoremstyle{definition}  \newtheorem{bsp}[thm]{Example}
\theoremstyle{definition}  
\theoremstyle{definition}

\DeclareMathOperator \id {id}

\DeclareMathOperator \uc {uc}
\DeclareMathOperator \tr {tr}

\newcommand{\I}{\mathds{1}}

\newcommand\fa{\qquad \text{for all \ }}

\newcommand\mc[1] {\mathcal{#1}}
\newcommand\mbb[1] {\mathds{#1}}

\newcommand{\eps}{\varepsilon}

\begin{document}

\title{On infinitesimal generators of sublinear Markov semigroups}
\author[F.~K\"{u}hn]{Franziska K\"{u}hn} 
\address[F.~K\"{u}hn]{Institut de Math\'ematiques de Toulouse, Universit\'e Paul Sabatier III Toulouse, 118 Route de Narbonne, 31062 Toulouse, France}
\email{franziska.kuehn1@tu-dresden.de}

\begin{abstract}
	We establish a Dynkin formula and a Courr\`ege-von Waldenfels theorem for sublinear Markov semigroups. In particular, we show that any sublinear operator $A$ on $C_c^{\infty}(\mbb{R}^d)$  satisfying the positive maximum principle can be represented as supremum of a family of pseudo-differential operators: \begin{equation*}
		Af(x) = \sup_{\alpha \in I} (-q_{\alpha}(x,D) f)(x).
	\end{equation*}
	As an immediate consequence, we obtain a representation formula for infinitesimal generators of sublinear Markov semigroups with a sufficiently rich domain. We give applications in the theory of non-linear Hamilton--Jacobi--Bellman equations and L\'evy processes for sublinear expectations.
\end{abstract}
\maketitle

\section{Introduction} \label{intro}

Let $(T_t)_{t \geq 0}$ be a Markov semigroup of linear operators on the space $\mc{B}_b(\mbb{R}^d)$ of bounded Borel measurable functions, i.e.\ a family of contractive linear operators $T_t: \mc{B}_b(\mbb{R}^d) \to \mc{B}_b(\mbb{R}^d)$ satisfying the semigroup property and the sub-Markov property ($0 \leq u \leq 1 \implies 0 \leq T_t u \leq 1$). Many properties of the semigroup $(T_t)_{t \geq 0}$ can be characterized via the associated infinitesimal generator \begin{equation*}
	Af(x):= \lim_{t \to 0} \frac{T_t f(x)-f(x)}{t}, \qquad f \in \mc{D}(A), \, x \in \mbb{R}^d,
\end{equation*}
whose domain $\mc{D}(A)$ is defined in such a way that the limit exists in a suitable sense, cf.\ Section~\ref{def}. Strongly continuous Markov semigroups are uniquely determined by their generator $(A,\mc{D}(A))$, cf.\ \cite[Corollary I.4.1.35]{jacob123}. If the domain $\mc{D}(A)$ of the infinitesimal generator is sufficiently rich, in the sense that the compactly supported smooth functions $f \in C_c^{\infty}(\mbb{R}^d)$ belong to $\mc{D}(A)$, then a result due to Courr\`ege \cite{courrege} and von Waldenfels \cite{wald1,wald2} states that $A|_{C_c^{\infty}(\mbb{R}^d)}$ has a representation of the form \begin{align*}
	Af(x) &= - c(x)f(x) + b(x) \cdot \nabla f(x) + \frac{1}{2} \tr(Q(x) \cdot \nabla^2 f(x)) \\ &\quad+ \int_{y \neq 0} (f(x+y)-f(x)-y \cdot \nabla f(x) \I_{(0,1)}(|y|)) \, \nu(x,dy), \qquad x \in \mbb{R}^d.
\end{align*}
Equivalently, $A|_{C_c^{\infty}(\mbb{R}^d)}$ can be written as a pseudo-differential operator with negative definite symbol, see Section~\ref{def} for details. In combination with Dynkin's formula \begin{equation*}
	T_t f- f = \int_0^t T_s Af \, ds, \qquad f \in \mc{D}(A), \, t \geq 0,
\end{equation*}
cf.\ \cite[Lemma I.4.1.14]{jacob123}, which can be seen as a counterpart of the fundamental theorem of calculus, this representation formula for $A$ has turned out to be a very powerful tool in the study of Markov semigroups and the associated Markov processes, cf.\ \cite{ltp,hoh,jacob123}. \par
In this paper, we extend Dynkin's formula and the Courr\`ege-Waldenfels theorem to \emph{sub}linear Markov semigroups $(T_t)_{t \geq 0}$, i.e.\ the operators $T_t$ are no longer assumed to be linear but only subadditive and positively homogeneous: \begin{equation*}
	\forall f,g, \, \forall \lambda \in [0,\infty)\::\: T_t (f+g) \leq T_t (f)+T_t(g) \qquad T_t (\lambda f) = \lambda T_t (f).
\end{equation*}
Sublinear Markov semigroups appear naturally in the study of stochastic processes on sublinear expectation spaces. They can be interpreted as stochastic processes under uncertainty, cf.\ Hollender \cite{julian}, and in many cases the semigroup has a representation of the form \begin{equation*}
	T_t f(x) = \sup_{\mbb{P} \in \mathfrak{P}^x} \mbb{E}_{\mbb{P}} f(X_t) := \sup_{\mbb{P} \in \mathfrak{P}^x} \int_{\Omega} f(X_t) \, d\mbb{P}
\end{equation*}
where the supremum is taken over a family of probability measures $\mathfrak{P}^x$ which depends on the starting point $x \in \mbb{R}^d$. As in case of classical Markov semigroups, it is possible to associate an evolution equation with sublinear semigroups $(T_t)_{t \geq 0}$, \begin{equation*}
	\frac{\partial}{\partial t} u(t,x) - A_x u(t,x)=0 \qquad u(0,x)=f(x),
\end{equation*}
where $A$ is the (sublinear) infinitesimal generator, cf.\ \cite[Proposition 4.10]{julian}. In a recent paper, Denk et al.\ \cite{denk19} studied under which conditions a sublinear semigroup is uniquely determined by its infinitesimal generator. The Courr\`ege-von Waldenfels theorem which we derive in this paper, cf.\ Corollary~\ref{lk-3}, shows that sublinear generators with a sufficiently rich domain have a representation of the form 
\begin{align*}
	Af(x) &= \sup_{\alpha \in I} \bigg( -c_{\alpha}(x) f(x) + b_{\alpha}(x) \cdot \nabla f(x) + \frac{1}{2} \tr(Q_{\alpha}(x) \cdot \nabla^2 f(x)) \\ &\quad \quad + \int_{y \neq 0} (f(x+y)-f(x)-y \cdot \nabla f(x) \I_{(0,1)}(|y|)) \, \nu_{\alpha}(x,dy)\bigg), \quad f \in C_c^{\infty}(\mbb{R}^d),
\end{align*}
and therefore sublinear Markov semigroups play an important role in the study of non-linear Hamilton--Jacobi--Bellman (HJB) equations, \begin{align} \label{hjb} \begin{aligned}
	\partial_t u(t,x) - \sup_{\alpha \in I} &\bigg(-c_{\alpha}(x) u(t,x)+ b_{\alpha}(x) \cdot \nabla_x u(t,x) + \frac{1}{2} \tr(Q_{\alpha}(x) \cdot \nabla_x^2 u(t,x))   \\  &+ \int_{y \neq 0} (u(t,x+y)-u(t,x)-y \cdot \nabla_x u(t,x) \I_{(0,1)}(|y|)) \, \nu_{\alpha}(x,dy)\bigg) = 0.
\end{aligned}
\end{align}
The idea to approach non-linear equations via stochastic processes on non-linear expectation spaces goes back to Peng \cite{peng} who introduced the so-called G-Brownian motion to study the G-heat equation \begin{equation*}
	\partial_t u(t,x)- \frac{1}{2} \sup_{\alpha \in I} (\tr(Q_{\alpha} \cdot \nabla_x^2 u(t,x)))=0.
\end{equation*}
More recently, the connection between non-linear integro-differential equations and nonlinear semigroups has been investigated in \cite{denk,julian,hjb,nendel19}. We will establish a general result which shows that for any ``nice'' sublinear semigroup $(T_t)_{t \geq 0}$ the mapping $u(t,x) := T_t f(x)$ is a viscosity solution to an HJB equation of the form \eqref{hjb}, cf.\ Section~\ref{visc}. \par
The paper is structured as follows. After introducing basic definitions and notation in Section~\ref{def}, we present a generalization of Dynkin's formula for sublinear Markov semigroups in Section~\ref{dynkin}. The Courr\`ege-von Waldenfels theorem for sublinear operators is stated and proved in Section~\ref{lk}. We use the Courr\`ege-von Waldenfels theorem to study the connection between HJB equations \eqref{hjb} and sublinear Markov semigroups, cf.\ Section~\ref{visc}. Some applications in the theory of L\'evy processes for sublinear expectations are presented in Section~\ref{levy}.

\section{Definitions and notation} \label{def}

\emph{Sublinear semigroups}: A family $\mc{H}$ of real-valued functions $f: \mbb{R}^d \to \mbb{R}$ is a \emph{convex cone} if $\alpha f+\beta g \in \mc{H}$ for all $\alpha,\beta \in \mbb{R}$, $f,g \in \mc{H}$ and $c \I_{\mbb{R}^d} \in \mc{H}$ for all $c \in \mbb{R}$. If $(T_t)_{t \geq 0}$ is a family of sublinear operators on a convex cone $\mc{H}$, i.e. \begin{equation*}
	\forall f,g \in \mc{H}, \, \lambda \in [0,\infty)\::\: T_t (f+g) \leq T_t f + T_t g \qquad T_t (\lambda f) = \lambda T_t f,
\end{equation*}
then we call $(T_t)_{t \geq 0}$ a \emph{sublinear Markov semigroup (on $\mc{H}$)} if the following properties are satisfied: \begin{enumerate}
	\item $T_{t+s} = T_t  T_s$ for all $s,t \geq 0$, and $T_0=\id$ (semigroup property),
	\item $f,g \in \mc{H}$, $f \leq g$ implies $T_t f \leq T_t g$ for all $t \geq 0$ (monotonicity),
	\item $T_t(\I_{\mbb{R}^d}) \leq \I_{\mbb{R}^d}$.
\end{enumerate}
If $T_t(\I_{\mbb{R}^d})=\I_{\mbb{R}^d}$, then $(T_t)_{t \geq 0}$ is \emph{conservative}. The \emph{(strong) infinitesimal generator} $(A,\mc{D}(A))$ of a sublinear Markov semigroup $(T_t)_{t \geq 0}$ is defined by \begin{align*}
	\mc{D}(A) &:= \left\{f \in \mc{H}; \, \exists g \in \mc{H}\::\: \lim_{t \to 0} \left\| \frac{T_t f-f}{t} - g \right\|_{\infty} = 0 \right\}, \\
	Af &:= \lim_{t \downarrow 0} \frac{T_tf-f}{t}, \qquad f \in \mc{D}(A).
\end{align*}
If the limit $g(x):=\lim_{t \to 0} t^{-1} (T_t f(x)-f(x))$ exists for all $x \in \mbb{R}^d$ and defines a function in $\mc{H}$, then $f$ is in the domain $\mc{D}(A^{(p)})$ of the \emph{pointwise infinitesimal generator} $A^{(p)}$, and we set \begin{equation*}
	A^{(p)}f(x) := \lim_{t \downarrow 0} \frac{T_t f(x)-f(x)}{t}, \qquad x \in \mbb{R}^d, \, f \in \mc{D}(A^{(p)}).
\end{equation*}
By definition, the pointwise infinitesimal generator $(A^{(p)},\mc{D}(A^{(p)})$ is an extension of the (strong) infinitesimal generator $(A,\mc{D}(A))$. It is immediate that $(A,\mc{D}(A))$ and $(A^{(p)},\mc{D}(A^{(p)}))$ are sublinear operators. The next lemma is simple to prove but will play an important role lateron when we investigate the structure of sublinear generators.

\begin{lem} \label{def-1}
	Let $(T_t)_{t \geq 0}$ be a sublinear Markov semigroup. The associated pointwise infinitesimal generator $(A^{(p)},\mc{D}(A^{(p)})$ satisfies the \emph{positive maximum principle}, i.e. \begin{equation*}
		f \in \mc{H}, \, f(x_0) = \sup_{x \in \mbb{R}^d} f(x) \geq 0 \implies A^{(p)}f (x_0) \leq 0.
	\end{equation*}
\end{lem}

\begin{proof}
	Fix $f \in \mc{H}$ and $x_0 \in \mbb{R}^d$ with $f(x_0) =\sup_{x \in \mbb{R}^d} f(x)$. Since $T_t$ is monotone, positively homogeneous and $T_t(1) \leq 1$, we have \begin{equation*}
		T_t f(x_0) \leq T_t (\|f\|_{\infty})(x_0) \leq \|f\|_{\infty} = f(x_0).
	\end{equation*}
	Subtracting $f(x_0)$ on both sides, dividing by $t>0$ and letting $t \downarrow 0$ yields $A^{(p)} f(x_0)\leq0.$
\end{proof}

The positive maximum principle clearly also holds for the (strong) generator $(A,\mc{D}(A))$. Our standard reference for non-linear semigroups is the monograph by Miyadera \cite{miyadera}. \par
\emph{Pseudo-differential operators}: Let $q(x,\cdot)$, $x \in \mbb{R}^d$ be a family of \emph{continuous negative definite} functions with representation \begin{equation}
	q(x,\xi) =c(x) -ib(x) \cdot \xi + \frac{1}{2} \xi \cdot Q(x) \xi + \int_{y \neq 0} (1-e^{iy \cdot \xi} + iy \cdot \xi \I_{(0,1)}(|y|)) \, \nu(x,dy), \qquad x,\xi \in \mbb{R}^d, \label{cndf}
\end{equation}
where $c(x) \geq 0$, $b(x) \in \mbb{R}^d$, $Q(x) \in \mbb{R}^{d \times d}$ is positive semidefinite and $\nu(x,dy)$ is a measure such that $\int_{y \neq 0} \min\{1,|y|^2\} \, \nu(x,dy)<\infty$ for each fixed $x \in \mbb{R}^d$. We will sometimes call $(c(x),b(x),Q(x),\nu(x,dy))$ \emph{characteristics} of $q(x,\cdot)$. The associated \emph{pseudo-differential operator} is defined on the smooth compactly supported functions $C_c^{\infty}(\mbb{R}^d)$ by \begin{equation}
	q(x,D) f(x) := - \int_{\mbb{R}^d} q(x,\xi) e^{ix \cdot \xi} \hat{f}(\xi) \, d\xi, \qquad f \in C_c^{\infty}(\mbb{R}^d), \, x \in \mbb{R}^d, \label{pseudo1}
\end{equation}
where $\hat{f}(\xi) := (2\pi)^{-d} \int_{\mbb{R}^d} f(x) e^{-ix \cdot \xi} \, dx$ is the Fourier transform of $f$, and $q$ is called \emph{symbol} of the operator. Equivalently, \begin{equation} \begin{aligned}
	q(x,D) f(x) &= -c(x) f(x) + b(x) \cdot \nabla f(x) + \frac{1}{2} \tr(Q(x) \cdot \nabla^2 f(x)) \\ &\quad + \int_{y \neq 0} \left( f(x+y)-f(x)-y \cdot \nabla f(x) \I_{(0,1)}(|y|) \right) \, \nu(x,dy).\end{aligned}  \label{pseudo2}
\end{equation}
An application of Taylor's formula shows that the pseudo-differential operator extends via \eqref{pseudo2} to an operator on $C_b^2(\mbb{R}^d)$ satisfying \begin{equation}
	|q(x,D)f(x)| \leq M \|f\|_{C_b^2(\mbb{R}^d)} \left( |c(x)| +|b(x)| + |Q(x)| + \int_{y \neq 0} \min\{1,|y|^2\} \, \nu(x,dy) \right), \,\, f \in C_b^2(\mbb{R}^d) \label{def-eq11}
\end{equation}
for some absolute constant $M>0$. Pseudo-differential operators appear naturally in the study of stochastic processes, e.g.\ Feller processes, stochastic differential equations and martingale problems, cf.\  \cite{ltp,hoh,jacob123,matters}. \par
\emph{Sublinear expectation spaces}: A \emph{sublinear expectation space} $(\Omega,\mc{H},\mc{E})$ consists of a set $\Omega \neq \emptyset$,  a linear space $\mc{H}$ of functions $f: \Omega \to \mbb{R}$ and a functional $\mc{E}: \mc{H} \to \real$ with the following properties: \begin{enumerate}
	\item $\mc{E}$ is subadditive, i.e.\ $\mc{E}(X+Y) \leq \mc{E}(X) + \mc{E}(Y)$ for all $X,Y \in \mc{H}$,
	\item $\mc{E}$ is positively homogeneous, i.e.\ $\mc{E}(\lambda X)=\lambda \mc{E}(X)$ for all $\lambda \geq 0$ and $X \in \mc{H}$,
	\item $\mc{E}$ preserves constants, i.e.\ $\mc{E}(c)=c$ for all $c \in \mbb{R}$,
	\item $\mc{E}$ is monotone, i.e.\ $\mc{E}(X) \leq \mc{E}(Y)$ for all $X,Y \in \mc{H}$ with $X \leq Y$.
\end{enumerate}
To introduce classical notions, such as random variables and independence, one needs to fix a class of test functions $\mc{T}$. In this paper, we take $\mc{T} := C_b^{\uc}(\mbb{R}^d)$, the space of bounded uniformly continuous functions. A ($\mbb{R}^d$-valued) \emph{random variable} on a sublinear expectation space $(\Omega,\mc{H},\mc{E})$ is a mapping $X: \Omega \to \mbb{R}^d$ such that $\varphi(X) \in \mc{H}$ for all $\varphi \in C_b^{\uc}(\mbb{R}^d)$. We also say that $X$ is \emph{adapted}. Two adapted $\mbb{R}^d$-valued random variables $X$ and $Y$ are \emph{equal in distribution}, $X \stackrel{d}{=} Y$, if \begin{equation*}
	\forall \varphi \in C_b^{\uc}(\mbb{R}^d)\::\: \mc{E}(\varphi(X)) = \mc{E}(\varphi(Y)).
\end{equation*}
The random variables $X$ and $Y$ are called \emph{independent} if \begin{equation*}
	\forall \varphi \in C_b^{\uc}(\mbb{R}^d \times \mbb{R}^d)\::\: \mc{E}(\varphi(X,Y)) = \mc{E} \big( \mc{E}(\varphi(x,Y)) \big|_{y=Y} \big);
\end{equation*}
it is implicitly assumed that all terms are well-defined. For an introduction to sublinear expectation spaces and their connection to stochastic processes on sublinear expectation spaces we refer to \cite{julian} and the references therein.

\emph{Function spaces}: The space of bounded Borel measurable functions $f: \mbb{R}^d \to \mbb{R}$ is denoted by $\mc{B}_b(\mbb{R}^d)$. We write $C_b(\mbb{R}^d)$ (resp.\ $C_b^{\uc}(\mbb{R}^d)$) for the bounded continuous (resp.\ uniformly continuous) functions $f: \mbb{R}^d \to \mbb{R}$. The compactly supported smooth functions $f: \mbb{R}^d \to\mbb{R}$ are denoted by $C_c^{\infty}(\mbb{R}^d)$. 

\section{Dynkin's formula for sublinear Markov semigroups} \label{dynkin}

Let $(T_t)_{t \geq 0}$ be a Markov semigroup of linear operators on $\mc{H}:=\mc{B}_b(\mbb{R}^d)$ with infinitesimal generator $(A,\mc{D}(A))$. Dynkin's formula states that \begin{equation}
	T_t f(x)- f(x) = \int_0^t T_s Af(x) \, ds, \qquad t \geq 0, \, x \in \mbb{R}^d, \label{dynkin-eq1}
\end{equation}
for all $f \in \mc{D}(A)$. If $T_t f(x) = \mbb{E}^x f(X_t)$ is the semigroup associated with a Markov process $(X_t)_{t \geq 0}$, then \eqref{dynkin-eq1} can be written equivalently in a probabilistic way: \begin{equation*}
	\mbb{E}^x f(X_t)-f(x) = \int_0^t \mbb{E}^x Af(X_s) \, ds, \qquad t \geq 0, \, x \in \mbb{R}^d.
\end{equation*}
Dynkin's formula \eqref{dynkin-eq1} holds more generally for functions in the domain of the weak generator, cf.\ Dynkin \cite{dynkin}, and for functions in the Favard space of order $1$, cf.\  Airault \& F\"{o}llmer \cite[p.~320-322]{foellmer74}; see \eqref{favard} below for the definition of the Favard space. In this section, we will show the following Dynkin-type formula for sublinear Markov semigroups \begin{equation*}
	- \int_0^t T_s Af(x) \, ds \leq T_t f(x)-f(x) \leq \int_0^t T_s Af(x) \, ds,  \qquad x \in \mbb{R}^d, \, t \geq 0,
\end{equation*}
for $f \in \mc{D}(A)$, see Theorem~\ref{dynkin-9} below. In general, the inequalities are strict. For the particular case that $(T_t)_{t \geq 0}$ is a Markov semigroup of linear operators, this gives the classical Dynkin formula \eqref{dynkin-eq1}. For the proof of the sublinear Dynkin formula, we need some auxiliary statements, see also Miyadera \cite[Section 3.1]{miyadera} for some related results.

\begin{lem} \label{dynkin-1}
	Let $(T_t)_{t \geq 0}$ be a sublinear Markov semigroup on $\mc{H}$ with Favard space $F_1$ of order $1$, i.e.\ \begin{equation}
		f \in F_1 \iff f \in \mc{H}, \, L(f) := \sup_{t>0} \frac{\|T_tf-f\|_{\infty}}{t} < \infty. \label{favard}
	\end{equation}
	Then: \begin{enumerate}
		\item\label{dynkin-1-i} $T_t (F_1) \subseteq F_1$ for all $t \geq 0$,
		\item\label{dynkin-1-ii} $\varphi(t) := L(T_t f)$ is non-increasing for each $f \in F_1$,
		\item\label{dynkin-1-iii} $t \mapsto T_t f(x)$ is globally Lipschitz continuous with Lipschitz constant $L(f)$ for all $f \in F_1$ and $x \in \mbb{R}^d$.
	\end{enumerate}
\end{lem}

\begin{proof}
	Fix $t>0$ and $f \in F_1$. Since $T_t$ is sublinear and monotone, we have \begin{equation*}
		T_s T_t f-T_t f = T_t T_s f-T_t f \leq T_t (T_s f-f) \leq \|T_s f-f\|_{\infty}
	\end{equation*}
	and \begin{equation*}
		- (T_s T_t f-T_t f) \leq T_t (f-T_s f) \leq \|T_sf-f\|_{\infty}.
	\end{equation*}
	Hence, $\varphi(t) =L(T_t f) \leq L(f)= \varphi(0)$. In particular, $T_t f \in F_1$ and \begin{equation*}
		\|T_{t+s} f-T_t f\|_{\infty} \leq L(T_t f) |s| \leq \varphi(0) |s|, \qquad s\geq 0. \qedhere
	\end{equation*}
\end{proof}
Since $\mc{D}(A) \subseteq F_1$, Lemma~\ref{dynkin-1} shows, in particular, that $t \mapsto T_t f(x)$ is Lipschitz continuous for all $f \in \mc{D}(A)$. It follows from Rademacher's theorem that there exists for each $x \in \mbb{R}^d$ some Lebesgue null set $N=N(x,f) \subseteq [0,\infty)$ such that the limit \begin{equation*}
	\lim_{s \to 0} \frac{T_{t+s} f(x)-T_t f(x)}{s}
\end{equation*}
exists for all $t \in [0,\infty) \backslash N$. For \emph{linear} strongly continuous Markov semigroups $(T_t)_{t \geq 0}$, it can be easily verified that the limit exists for \emph{all} $t \geq 0$ uniformly in $x \in \mbb{R}^d$, and so $T_t(\mc{D}(A)) \subseteq \mc{D}(A)$. This is no longer true for sublinear semigroups: there may be functions $f \in \mc{D}(A)$ such that $T_t f \in \mc{D}(A)$ fails to hold for $t>0$. 

\begin{bsp} \label{dynkin-2} 
	The family of operators \begin{equation*}
		T_t f(x) := \sup_{|s| \leq 1} f(x+s) = \sup_{b \in [-1,1]} f(x+bt), \quad t \geq 0,
	\end{equation*}
	defines a strongly continuous sublinear Markov semigroup on  $\mc{H} := C_b^{\uc}(\mbb{R})$. Moreover, it follows from Taylor's formula that \begin{equation*}
		\lim_{t \to 0} \frac{T_t f(x)-f(x)}{t} = |f'(x)| 
	\end{equation*}
	for all $f \in C_b^{2}(\mbb{R})$, and the convergence is uniformly in $x \in \mbb{R}$. Thus, $C_b^{2}(\mbb{R}) \subseteq \mc{D}(A)$ and $Af = |f'|$ for $f \in C_b^{2}(\mbb{R})$. Take a function $f \in C_b^2(\mbb{R}) \subseteq \mc{D}(A)$ such that $f(x) \in [0,1]$ for all $x \in [-1,1]$, $f(x)=-x$ for $x \in [-2,-1]$ and $f(x)=1$ for $1/2 \leq x \leq 2$. Then \begin{equation*}
		T_t f(0) = \begin{cases} 1, & t \in [1/2,1], \\ t, & t \in [1,2], \end{cases}
	\end{equation*}
	and so $t \mapsto T_t f(0)$ is not differentiable at $t=1$, i.e. $T_1 f \notin \mc{D}(A)$.
\end{bsp}

Let us remark that Denk et al.\ \cite{denk19} showed very recently that the generator of a sublinear semigroup (on a ``nice'' space)  can be extended in such a way that the domain of the extended generator is invariant under $T_t$. If the semigroup is continuous from above, then the extended generator uniquely characterizes the semigroup. \par \medskip

Though $(T_{t+s}f-T_t f)/s$ does not necessarily converge as $s \to 0$, we can show that difference quotient is bounded from above (resp.\ below) by $T_t Af$ (resp.\ $-T_t(-Af)$). This bound for the slope of $t \mapsto T_t f$ is the key for the proof of Dynkin's formula.

\begin{prop} \label{dynkin-3}
	Let $(T_t)_{t \geq 0}$ be a sublinear Markov semigroup on $\mc{H}$ with strong infinitesimal generator $(A,\mc{D}(A))$. If $f \in \mc{D}(A)$, then \begin{equation}
		- T_t (-Af) \leq \liminf_{s \to 0} \frac{T_{t+s}-T_t f}{s} \leq \limsup_{s \to 0} \frac{T_{t+s} f-T_t f}{s} \leq  T_t Af \label{dynkin-eq5}
	\end{equation}
	for all $t \geq 0$. 
\end{prop}

\begin{proof}
	Fix $f \in \mc{D}(A)$ and $t \geq 0$. By the semigroup property and subadditivity of $(T_t)_{t \geq 0}$, we have \begin{equation*}
		\frac{T_{t+s} f-T_t f}{s} - T_t Af \leq T_t \left( \frac{T_s f-f}{s} - Af \right).
	\end{equation*}
	Since $T_t$ is monotone and $T_t 1\leq 1$, this gives \begin{equation*}
		\frac{T_{t+s} f-T_t f}{s}- T_t Af  \leq \left\| \frac{T_s f-f}{s}-Af \right\|_{\infty} \xrightarrow[]{s \to 0} 0.
	\end{equation*}
	Hence, \begin{equation*}
		\limsup_{s \to 0} \frac{T_{t+s}f-T_t f}{s} \leq T_t Af.
	\end{equation*}
	On the other hand, it follows from the subadditivity of $T_t$ that \begin{equation*}
		T_t f \leq T_t (f-T_s f) + T_t T_s f,
	\end{equation*}
	and so \begin{equation*}
		\frac{T_{t+s} f-T_t f}{s} + T_t (-Af) 
		\geq - \frac{T_t (f-T_s f)}{s} + T_t (-Af).
	\end{equation*}
	Using \begin{equation*}
		T_t \left( \frac{f-T_s f}{s} \right) \leq T_t \left( \frac{f-T_s f}{s} + Af \right) + T_t (-Af)
	\end{equation*}
	we find that \begin{equation*}
		\frac{T_{t+s} f-T_t f}{s} + T_t (-Af) 
		\geq - T_t \left( \frac{f-T_s f}{s} + Af \right) 
		\geq - \left\| \frac{f-T_s f}{s} + Af \right\|_{\infty}
		\xrightarrow[]{s \to 0} 0. \qedhere
	\end{equation*}
\end{proof}

\begin{thm}[Dynkin's formula] \label{dynkin-9}
	Let $(T_t)_{t \geq 0}$ be a sublinear Markov semigroup on $\mc{H}$ with strong infinitesimal generator $(A,\mc{D}(A))$. If $f \in \mc{D}(A)$, then \begin{equation}
			- \int_0^t T_s (-Af) \, ds \leq T_t f-f \leq \int_0^t T_s (Af) \, ds \fa t \geq 0. \label{dynkin-eq21}
		\end{equation}
\end{thm}

\begin{proof}
	Fix $f \in \mc{D}(A)$. By Lemma~\ref{dynkin-1}, $t \mapsto T_t f(x)$ is globally Lipschitz continuous for all $x \in \mbb{R}^d$, and therefore it follows from Rademacher's theorem that there exists a mapping $g(t,x)$ such that \begin{equation*}
		T_t f(x)-f(x) = \int_0^t g(s,x) \, ds
	\end{equation*}
	and \begin{equation*}
		\frac{d}{dt} T_t f(x) =g(t,x)
	\end{equation*}
	for Lebesgue almost every $t \geq 0$ (the exceptional null set may depend on $x \in \mbb{R}^d$). Dynkin's formula \eqref{dynkin-eq21} is now an immediate consequence of Proposition~\ref{dynkin-3}.
\end{proof}

\begin{bem} \label{dynkin-11} \begin{enumerate}[wide, labelwidth=!, labelindent=0pt]
	\item If $(T_t)_{t \geq 0}$ is a Markov semigroup of \emph{linear} operators, then we recover the classical Dynkin formula: \begin{equation*}
		T_t f-f = \int_0^t T_s Af \, ds, \quad f \in \mc{D}(A),\, t \geq 0.
	\end{equation*}
	\item In general, the inequalities in \eqref{dynkin-eq21} are strict. Consider, for instance, \begin{equation*}
		T_t f(x) = \sup_{b \in [-1,1]} f(x+bt) \qquad Af = |f'|
	\end{equation*}
	(see Example~\ref{dynkin-2}), then \begin{equation*}
		T_t f(x)-f(x) = \sup_{b \in [-1,1]} \int_0^{bt} f'(x+r) \, dr
	\end{equation*}
	is, in general, strictly smaller than \begin{equation*}
		\int_0^t T_s Af(x) \, ds = \int_0^t \sup_{b \in [-1,1]} |f'(x+bs)| \, ds,
	\end{equation*}
	e.g.\ if $f'$ has strict maximum in $x$. 
	\item In Section~\ref{lk} we will identify $A|_{C_c^{\infty}(\mbb{R}^d)}$ under the assumption that $C_c^{\infty}(\mbb{R}^d) \subseteq \mc{D}(A)$. In combination with Dynkin's formula \eqref{dynkin-eq21}, this gives a useful tool to establish probability estimate for Markov processes on sublinear spaces, e.g.\ estimates for fractional moments. 
\end{enumerate}
\end{bem}

As a direct consequence of Dynkin's formula we obtain the following corollary; see e.g.\ \cite[Lemma 2.3]{rs-cons} for the counterpart in the framework of linear semigroups.

\begin{kor} \label{dynkin-13}
	Let $(T_t)_{t \geq 0}$ be a sublinear Markov semigroup on $\mc{H}$ with strong infinitesimal generator $(A,\mc{D}(A))$. If $f \in \mc{D}(A)$, then \begin{equation}
		\|Af\|_{\infty} = \sup_{t>0} \frac{\|T_t f-f\|_{\infty}}{t}. \label{dynkin-eq25}
	\end{equation}
	In particular, \begin{equation}
		\|T_{t} f-T_s f\|_{\infty} \leq \|Af\|_{\infty} |t-s|, \qquad s,t \geq 0. \label{dynkin-eq27}
	\end{equation}
\end{kor}

\begin{proof}
	Since $T_t$ is monotone for each $t \geq 0$, it follows from Dynkin's formula \eqref{dynkin-eq21} that \begin{equation*}
		\|T_t f-f\|_{\infty} \leq t \|Af\|_{\infty}.
	\end{equation*}
	On the other hand, the very definition of $A$ gives \begin{equation*}
		\|Af\|_{\infty} = \lim_{t \to 0} \frac{\|T_t f-f\|_{\infty}}{t},
	\end{equation*}
	and this proves \eqref{dynkin-eq25}. From Lemma~\ref{dynkin-1}\eqref{dynkin-1-ii}, we get the Lipschitz estimate \eqref{dynkin-eq27}.
\end{proof}
	
\section{Courr\`ege-von Waldenfels theorem for sublinear operators} \label{lk}

Let $A: \mc{D}(A) \to \mc{B}_b(\mbb{R}^d)$ be a linear operator satisfying the positive maximum principle, i.e. \begin{equation}
	f \in \mc{D}(A), \, f(x_0) = \sup_{x \in \mbb{R}^d} f(x) \geq 0 \implies Af(x_0) \leq 0. \label{pmp}
\end{equation}
If $C_c^{\infty}(\mbb{R}^d) \subseteq \mc{D}(A)$, then $A|_{C_c^{\infty}(\mbb{R}^d)}$ has a representation of the form \begin{align*}\begin{aligned}
	Af(x) &= -c(x) f(x) + b(x) \cdot \nabla f(x) + \frac{1}{2} \tr(Q(x) \cdot \nabla^2 f(x)) \\ &\quad + \int_{y \neq 0} \left( f(x+y)-f(x)-y \cdot \nabla f(x) \I_{(0,1)}(|y|) \right) \, \nu(x,dy)\end{aligned} \end{align*}
where $c(x) \geq 0$, $b(x) \in \mbb{R}^d$, $Q(x) \in \mbb{R}^{d \times d}$ is positive semidefinite and $\nu(x,dy)$ is a measure such that $\int_{y \neq 0} \min\{1,|y|^2\} \, \nu(x,dy)<\infty$; this result is due to Courr\`ege \cite{courrege} and von Waldenfels \cite{wald1,wald2}. Since infinitesimal generators of Markov processes satisfy the positive maximum principle, this gives immediately a representation formula for infinitesimal generators with a sufficiently rich domain, cf.\ \cite[Theorem 2.21]{ltp} or \cite{jacob123}. Recently, the result by Courr\`ege and von Waldenfels was generalized to Lie groups, cf.\ \cite{apple19}. In this section, we establish the following Courr\`ege-von Waldenfels theorem for \emph{sub}linear operators satisfying the positive maximum principle.  

\begin{thm} \label{lk-1}
	Let $A: \mc{D}(A) \to \mc{B}_b(\mbb{R}^d)$ be a sublinear operator with $\mc{D}(A) \subseteq \mc{B}_b(\mbb{R}^d)$. Assume that $A$ satisfies the positive maximum principle \eqref{pmp}. If $C_c^{\infty}(\mbb{R}^d) \subseteq \mc{D}(A)$, then there exist an index set $I$ and a family $(c_{\theta}(x),b_{\theta}(x),Q_{\theta}(x),\nu_{\theta}(x,dy))$, $\theta \in I$, $x \in \mbb{R}^d$, of characteristics such that \begin{align} \label{lk-eq6}
		Af(x) &= \sup_{\theta \in I} A_{\theta} f(x), \qquad f \in C_c^{\infty}(\mbb{R}^d), \, x \in \mbb{R}^d,
	\end{align}
	where \begin{align*} \begin{aligned} A_{\theta} f(x) &:= - c_{\theta}(x) f(x) + b_{\theta}(x) \cdot \nabla f(x) + \frac{1}{2} \tr(Q_{\theta}(x) \cdot \nabla^2 f(x)) \\ &\qquad + \int_{y \neq 0} (f(x+y)-f(x)-y \cdot \nabla f(x) \I_{(0,1)}(|y|)) \, \nu_{\theta}(x,dy). \end{aligned}
	\end{align*}
	The supremum is attained, i.e.\ for any $f \in C_c^{\infty}(\mbb{R}^d)$ and $x \in \mbb{R}^d$ there exists $\theta=\theta(f,x)$ such that $Af(x) = A_{\theta} f(x)$. For each $x \in \mbb{R}^d$ the family $(c_{\theta}(x),b_{\theta}(x),Q_{\theta}(x),\nu_{\theta}(x,dy))$, $\theta \in I$, is uniformly bounded, i.e. \begin{equation}
			\sup_{\theta \in I} \left( |c_{\theta}(x)|  + |b_{\theta}(x)| + |Q_{\theta}(x)| + \int_{y \neq 0} \min\{1,|y|^2\} \, \nu_{\theta}(x,dy) \right)<\infty. \label{lk-eq7}
		\end{equation}
\end{thm}

Since the generator of a sublinear Markov semigroup satisfies the positive maximum principle, Theorem~\ref{lk-1} gives, in particular, a representation formula for sublinear generators whose domains contain $C_c^{\infty}(\mbb{R}^d)$, see Corollary~\ref{lk-3} below. If $A$ is a \emph{linear} operator, then the index set $I$ consists of a single element, and we recover the classical Courr\`ege-von Waldenfels theorem. \par
To prove Theorem~\ref{lk-1}, we need a representation result for sublinear functionals. We say that a functional $B: \mc{D} \to \mbb{R}$ defined on a subspace $\mc{D}$ of functions $f: \mbb{R}^d \to \mbb{R}$ \emph{satisfies the positive maximum principle in $x_0 \in \mbb{R}^d$} if $f \in \mc{D}$, $f(x_0)= \sup_{x \in \mbb{R}^d} f(x) \geq 0$ implies $Bf \leq 0$.

\begin{lem} \label{lk-5}
	Let $B: \mc{D} \to \mbb{R}$ be a sublinear functional on a linear space $\mc{D} \subseteq \mc{B}_b(\mbb{R}^d)$. If  $B$ satisfies the positive maximum principle in some point $x_0 \in \mbb{R}^d$,  then there exists a family $(B_{\theta})_{\theta \in \Theta}$ of linear functionals on $\mc{D}$ satisfying the positive maximum principle in $x_0$ such that \begin{equation*}
		Bf = \sup_{\theta \in \Theta} B_{\theta} f, \qquad f \in \mc{D}.
	\end{equation*}
	The supremum is attained, i.e.\ for every $f \in \mc{D}$ there exists some $\theta=\theta(f) \in \Theta$ such that $Bf = B_{\theta} f$. 
\end{lem}

It was shown in \cite[Theorem 3.5]{julian} that any sublinear functional $B$ on a linear space has a representation of the form $Bf = \sup_{\theta \in \Theta} B_{\theta} f$ for a family of linear functionals. For our application it is crucial to have the positive maximum principle for $B_{\theta}$.

\begin{proof}[Proof of Lemma~\ref{lk-5}]
	Set $\Theta := \{\theta: \mc{D} \to \mbb{R}; \theta$ is linear and $\theta \leq B\}$ and $B_{\theta} := \theta$ for $\theta \in \Theta$. An application of the Hahn-Banach theorem shows that \begin{equation*}
		Bf = \sup_{\theta \in \Theta} B_{\theta} f, \qquad f \in \mc{D},
	\end{equation*}
	and the supremum is attained, see \cite[Proof of Theorem 3.5]{julian} for details. Now let $f \in \mc{D}$ such that $f(x_0) = \sup_{x \in \mbb{R}^d} f(x) \geq 0$. By assumption, \begin{equation*}
		0 \geq Bf = \sup_{\theta \in \Theta} B_{\theta} f,
	\end{equation*}
	and so $B_{\theta} f \leq 0$ for all $\theta \in \Theta$, i.e.\ $B_{\theta}$ satisfies the positive maximum principle in $x_0 \in \mbb{R}^d$.
\end{proof}

\begin{proof}[Proof of Theorem~\ref{lk-1}] 
	\textbf{Step 1:} Throughout this first part of the proof, we fix $x \in \mbb{R}^d$. On a linear subspace $\mc{D}$ of $\mc{D}(A)$ define a sublinear operator $B: \mc{D} \to \mbb{R}$ by $Bf := Af(x)$. From the positive maximum principle for $A$, it is immediate that $B$ satisfies the positive maximum principle in $x$. Applying Lemma~\ref{lk-5}, we find that there exist an index set $\Theta=\Theta(x)$ and a family $(A_{\theta})_{\theta \in \Theta}$ of linear functionals on $\mc{D}$ satisfying the positive maximum principle in $x \in \mbb{R}^d$ such that \begin{equation}
		Af(x) = \sup_{\theta \in \Theta} A_{\theta} f, \qquad f \in \mc{D}. \label{lk-eq15}
	\end{equation}
	By assumption, we can choose $\mc{D} := C_c^{\infty}(\mbb{R}^d)$. Since the operators $A_{\theta}$ are linear, the classical Courr\`ege-von Waldenfels theorem shows  that there exist $c_{\theta} \geq 0$, $b_{\theta} \in \mbb{R}^d$, a positive semidefinite matrix $Q_{\theta} \in \mbb{R}^{d \times d}$ and a measure $\nu_{\theta}$ with $\int_{y \neq 0} \min\{1,|y|^2\} \, \nu_{\theta}(dy)<\infty$ such that \begin{equation*}
		A_{\theta} f = - c_{\theta} f(x) + b_{\theta}\cdot \nabla f(x) + \frac{1}{2} \tr(Q_{\theta} \cdot \nabla^2 f(x)) + \int_{y \neq 0} (f(x+y)-f(x)-y \cdot \nabla f(x) \I_{(0,1)}(|y|)) \, \nu_{\theta}(dy)
	\end{equation*}
	for all $f \in C_c^{\infty}(\mbb{R}^d)$, see also \cite[Theorem 3.4]{apple19}. Next we prove that the family $(c_{\theta},b_{\theta},Q_{\theta},\nu_{\theta})$, $\theta \in \Theta$, is bounded, i.e. \begin{equation}
		\sup_{\theta \in \Theta} \left( c_{\theta} + |b_{\theta}| + |Q_{\theta}| + \int_{y \neq 0} \min\{1,|y|^2\} \, \nu_{\theta}(dy) \right)<\infty. \tag{$\star$} \label{lk-eq-st1}
	\end{equation}
	Pick $\chi \in C_c^{\infty}(\mbb{R}^d)$ such that $\I_{B(x,1/2)} \leq \chi \leq \I_{B(x,1)}$. From $\chi(x)=1$, $\nabla \chi(x)=0$ and $\nabla^2 \chi(x)=0$, we find that \begin{align*}
		A_{\theta}(-\chi)=  c_{\theta} +  \int_{y \neq 0} (1-\chi(x+y)) \, \nu_{\theta}(dy)
		\geq c_{\theta} + \int_{|y| \geq 1} \nu_{\theta}(dy).
	\end{align*}
	Hence, \begin{align*}
		\sup_{\theta \in \Theta} \left( c_{\theta} + \int_{|y| \geq 1} \nu_{\theta}(dy) \right) \leq \sup_{\theta \in \Theta} A_{\theta}(-\chi) = A(-\chi)(x)<\infty.
	\end{align*}
	For fixed $j \in \{1,\ldots,d\}$ consider the mapping $f(y) := (y^{(j)}-x^{(j)})^2 \chi(y)$, then \begin{equation*}
		A_{\theta} f = Q_{\theta}^{(j,j)} + \int_{y \neq 0} (y^{(j)})^2 \chi(y+x) \, \nu_{\theta}(dy) \geq Q_{\theta}^{(j,j)} + \int_{0<|y|<1} (y^{(j)})^2 \, \nu_{\theta}(dy)
	\end{equation*}
	for all $\theta \in \Theta$. Since $A_{\theta} f$ is bounded from above by $Af(x)$, this implies that the right-hand side is bounded uniformly in $\theta \in \Theta$. In a similar fashion, we consider $y \mapsto (y^{(j)}-x^{(j)}) \chi(y)$ and $y \mapsto (y^{(j)}-x^{(j)}) (y^{(i)}-x^{(i)}) \chi(y)$ to obtain that \begin{equation*}
		\sup_{\theta \in \Theta} \left( \sum_{j=1}^d |b_{\theta}^{(j)}| + \sum_{i \neq j} |Q_{\theta}^{(i,j)}| \right) < \infty.
	\end{equation*}
	Combining the above estimates, we get \eqref{lk-eq-st1}. \par
	\textbf{Conclusion:} The previous step shows that for any $x \in \mbb{R}^d$ there exist an index set $\Theta=\Theta(x)$ and a uniformly bounded family  $(c_{\theta}(x),b_{\theta}(x),Q_{\theta}(x),\nu_{\theta}(x,dy))$ of characteristics such that \begin{align*}
		Af(x) = \sup_{\theta \in \Theta(x)} &\bigg( -c_{\theta}(x) f(x) + b_{\theta}(x) \cdot \nabla f(x) + \frac{1}{2} \tr(Q_{\theta}(x) \cdot \nabla^2 f(x)) \\
		&+ \int_{y \neq 0} (f(x+y)-f(x)-y \cdot \nabla f(x) \I_{(0,1)}(|y|)) \, \nu_{\theta}(x,dy) \bigg)
	\end{align*}
	for all $f \in C_c^{\infty}(\mbb{R}^d)$. If we define an index set $I$ by $I := \bigcup_{x \in \mbb{R}^d} \Theta(x)$ and set \begin{equation*}
		(c_{\theta}(x),b_{\theta}(x),Q_{\theta}(x),\nu_{\theta}(x,dy))
		:= (c_{\theta'}(x),b_{\theta'}(x),Q_{\theta'}(x),\nu_{\theta'}(x,dy)), \qquad \theta \in I \backslash \Theta(x)
	\end{equation*}
	for some fixed $\theta' \in \Theta(x)$, then we obtain the representation 
	\begin{align*}
		Af(x) = \sup_{\theta \in I} &\bigg(-c_{\theta}(x) f(x) +  b_{\theta}(x) \cdot \nabla f(x) + \frac{1}{2} \tr(Q_{\theta}(x) \cdot \nabla^2 f(x)) \\
		&+ \int_{y \neq 0} (f(x+y)-f(x)-y \cdot \nabla f(x) \I_{(0,1)}(|y|)) \, \nu_{\theta}(x,dy) \bigg). \qedhere
	\end{align*}
\end{proof}

\begin{kor} \label{lk-3}
	Let $(T_t)_{t \geq 0}$ be a sublinear Markov semigroup on $\mc{H}$ with pointwise infinitesimal generator $(A^{(p)},\mc{D}(A^{(p)}))$. If $C_c^{\infty}(\mbb{R}^d) \subseteq \mc{D}(A^{(p)})$, then there exists a family $q_{\theta}(x,\cdot)$, $\theta \in I$, $x \in \mbb{R}^d$, of continuous negative definitions functions such that \begin{equation}
		A^{(p)} f(x) = \sup_{\theta \in I} (-q_{\theta}(x,D)f(x)) \fa f \in C_c^{\infty}(\mbb{R}^d),\, x \in \mbb{R}^d, \label{lk-eq5}
	\end{equation}
	cf.\ \eqref{pseudo2}. Equivalently, \begin{align}  \label{lk-eq55}\begin{aligned}
		A^{(p)} f(x) = \sup_{\theta \in I} &\bigg( -c_{\theta}(x) f(x) + b_{\theta}(x) \cdot \nabla f(x) + \frac{1}{2} \tr(Q_{\theta}(x) \cdot \nabla^2 f(x)) \\ &\qquad + \int_{y \neq 0} (f(x+y)-f(x)-y \cdot \nabla f(x) \I_{(0,1)}(|y|)) \, \nu_{\theta}(x,dy) \bigg) \end{aligned}
				\end{align}
	where $(c_{\theta}(x),b_{\theta}(x),Q_{\theta}(x),\nu_{\theta}(x,\cdot))$ is the characteristics associated via the L\'evy--Khintchine representation \eqref{cndf} with $q_{\theta}(x,\cdot)$. The family $(c_{\theta}(x),b_{\theta}(x),Q_{\theta}(x),\nu_{\theta}(x,dy))$, $\theta \in I$, is uniformly bounded for each $x \in \mbb{R}^d$: \begin{equation*}
		\sup_{\theta \in I} \left( c_{\theta}(x)+ |b_{\theta}(x)|+ |Q_{\theta}(x)|+ \int_{y \neq 0} \min\{1,|y|^2\} \, \nu_{\theta}(x,dy) \right)  < \infty.
	\end{equation*}
\end{kor}

\begin{proof}
	By Theorem~\ref{lk-1} and the positive maximum principle for sublinear generators, cf.\ Lemma~\ref{def-1}, the generator $A^{(p)}$ has a representation of the form \eqref{lk-eq6}. Equivalently, \begin{equation*}
		A^{(p)} f(x) = \sup_{\theta \in I} (-q_{\theta}(x,D) f)(x), \qquad f \in C_c^{\infty}(\mbb{R}^d),
	\end{equation*}
	where $q_{\theta}(x,\cdot)$ is the continuous negative definite function which is associated via the L\'evy--Khintchine representation with $(c_{\theta}(x),b_{\theta}(x),Q_{\theta}(x),\nu_{\theta}(x,dy))$, cf.\ \eqref{cndf} and \eqref{pseudo1}. 
\end{proof}

\begin{bem} \label{lk-7} \begin{enumerate}[wide, labelwidth=!, labelindent=0pt]
	\item The representation \eqref{lk-eq5} is, in general, not unique. For instance, the operator \begin{equation*}
		Af(x) = |f'(x)| = \sup\{f'(x),-f'(x)\}
	\end{equation*}
	(see Example~\ref{dynkin-2}) has at least two representations of the form \eqref{lk-eq5}: \begin{equation*}
		Af(x) = \sup_{\theta \in [-1,1]} (\theta f'(x)) = \sup_{\theta \in \{-1,1\}} (\theta f'(x)).
	\end{equation*}
	\item If $(T_t)_{t \geq 0}$ is a conservative sublinear Markov semigroup, then $q_{\theta}(x,0)=0$ and $c_{\theta}(x)=0$ for all $\theta \in I$ and $x \in \mbb{R}^d$, see Corollary~\ref{lk-125} below.
		\item In the statement of Corollary~\ref{lk-3} we may replace the pointwise generator $A^{(p)}$ by the strong generator $A$. This follows from the fact that $(A^{(p)},\mc{D}(A^{(p)}))$ extends $(A,\mc{D}(A))$. 
	\item Sufficient conditions which ensure that $C_c^{\infty}(\mbb{R}^d) \subseteq \mc{D}(A^{(p)})$ were obtained in \cite{hjb,julian,nendel19}.
\end{enumerate} \end{bem}

For translation-invariant Markov semigroups, i.e.\ semigroups satisfying $(T_t f)(x) = (T_t f(x+\cdot))(0)$, the representation \eqref{lk-eq5} simplifies since the family of continuous negative definite functions does not depend on $x \in \mbb{R}^d$.

\begin{kor} \label{lk-9}
	Let $(T_t)_{t \geq 0}$ be a translation invariant sublinear Markov semigroup on $\mc{H}$ with pointwise generator $(A^{(p)},\mc{D}(A^{(p)}))$. If $C_c^{\infty}(\mbb{R}^d) \subseteq \mc{D}(A)$, then there exists a family $(\psi_{\theta})_{\theta \in I}$ of continuous negative definite functions such that \begin{equation*}
		A^{(p)} f(x) = \sup_{\theta \in I} (-\psi_{\theta}(D) f(x)) \fa f \in C_c^{\infty}(\mbb{R}^d),\, x \in \mbb{R}^d.
	\end{equation*}
	The associated family $(c_{\theta},b_{\theta},Q_{\theta},\nu_{\theta})$, $\theta \in I$, of triplets satisfies \begin{equation*}
		\sup_{\theta \in I} \left( c_{\theta}+|b_{\theta}| + |Q_{\theta}| +\int_{y \neq 0} \min \left\{1, |y|^2 \right\} \, \nu_{\theta}(dy) \right) < \infty.
	\end{equation*}
\end{kor}

\begin{proof}
	By Corollary~\ref{lk-3}, there exists a family $q_{\theta}(x,\cdot)$ of continuous negative definite functions such that \begin{equation*}
		A^{(p)} f(x) = \sup_{\theta \in I} (-q(x,D) f)(x), \qquad f \in C_c^{\infty}(\mbb{R}^d), \, x \in \mbb{R}^d.
	\end{equation*}
	Since the semigroup is translation invariant, it follows from the definition of the generator that $A^{(p)} f(x) = (A^{(p)} f(x+\cdot))(0)$ for all $f \in \mc{D}(A^{(p)})$, $x \in \mbb{R}^d$. Thus, \begin{equation*}
		A^{(p)} f(x) = \sup_{\theta \in I} (-\psi_{\theta}(D) f)(x), \qquad f \in C_c^{\infty}(\mbb{R}^d), \, x \in \mbb{R}^d,
	\end{equation*}
	for $\psi_{\theta} := q_{\theta}(0,\cdot)$. The uniform boundedness of the associated triplets is evident from Corollary~\ref{lk-3}.
\end{proof}

\begin{bem}
	If $(T_t)_{t \geq 0}$ is a \emph{linear} translation invariant Markov semigroup, say on $C_b(\mbb{R}^d)$, then $(T_t)_{t \geq 0}$ is the semigroup of a L\'evy process, and therefore the assumption $C_c^{\infty}(\mbb{R}^d) \subseteq \mc{D}(A^{(p)})$ in Corollary~\ref{lk-9} is automatically satisfied, see e.g.\ \cite[Lemma 6.3]{barcelona}.  This is no longer true for sublinear semigroups. For instance, $T_0 := \id$, \begin{equation*}
		T_t f(x) := \|f\|_{\infty} = \sup_{y \in \mbb{R}^d} |f(y)|, \qquad t>0,
	\end{equation*}
	defines a translation invariant sublinear Markov semigroup, but $C_c^{\infty}(\mbb{R}^d)$ is not contained in the domain $\mc{D}(A^{(p)})$ of the pointwise generator. In fact, even pointwise convergence $T_t f(x) \xrightarrow[]{t \to 0} f(x)$ fails to hold for $f \in C_c^{\infty}(\mbb{R}^d)$. \par
	 In view of this example, it is natural to ask whether $C_c^{\infty}(\mbb{R}^d) \subseteq \mc{D}(A^{(p)})$ holds under the additional assumption that $(T_t)_{t \geq 0}$ is strongly continuous (on a sufficiently large domain). This seems to be an open question.
\end{bem}

For many applications it would be useful to have the representation $A^{(p)} f = \sup_{\theta} (-q_{\theta}(x,D)f)$ from Corollary~\ref{lk-3} not only for $f \in C_c^{\infty}(\mbb{R}^d)$, but for a larger class of functions. In the linear framework, one typically invokes the closedness of the infinitesimal generator to extend the representation formula e.g.\ to $f \in C_c^2(\mbb{R}^d)$, cf.\ \cite[Corollary 3.8]{schnurr12} or \cite[Theorem 2.37]{ltp}. The situation is more complicated for sublinear infinitesimal generators. In our next result we give an extension to smooth functions with bounded derivatives. It will play a crucial role when we study viscosity solutions to non-linear Cauchy problems, see Corollary~\ref{lk-10}.

\begin{kor} \label{lk-8}
	Let $(T_t)_{t \geq 0}$ be a sublinear Markov semigroup on $\mc{H}$ with pointwise infinitesimal generator $(A^{(p)},\mc{D}(A^{(p)}))$. Assume that $C_c^{\infty}(\mbb{R}^d) \subseteq \mc{D}(A^{(p)})$. By Corollary~\ref{lk-3}, \begin{equation}
		A^{(p)} f(x) = \sup_{\theta \in I} (-q_{\theta}(x,D) f)(x), \qquad f \in C_c^{\infty}(\mbb{R}^d),\, x \in \mbb{R}^d, \label{lk-eq27}
	\end{equation}
	for a family of continuous negative definite functions $q_{\theta}(x,\cdot)$ with characteristics $(0,b_{\theta}(x),Q_{\theta}(x),\nu_{\theta}(x,\cdot))$, $\theta \in I$. 
	\begin{enumerate}
		\item\label{lk-8-i} If the family $\nu_{\theta}(x,\cdot)$, $\theta \in I$, is tight for each $x \in \mbb{R}^d$, then $C_b^{\infty}(\mbb{R}^d) \subseteq \mc{D}(A^{(p)})$.
		\item\label{lk-8-ii} If $\mc{D} \subseteq C_b^{\infty}(\mbb{R}^d)$ is a linear subspace such that $\mc{D} \subseteq \mc{D}(A^{(p)})$, then \eqref{lk-eq27} holds for any $f \in \mc{D}$.
	\end{enumerate}
\end{kor}

\begin{bem} \label{lk-85} \begin{enumerate}[wide, labelwidth=!, labelindent=0pt]
	\item\label{lk-85-i} If $(T_t)_{t \geq 0}$ is a linear Markov semigroup, then the index set $I$ consists of a single element, and therefore the tightness of $\nu_{\theta}(x,\cdot)$, $\theta \in I$, is  automatically satisfied for each fixed $x \in \mbb{R}^d$.
	\item\label{lk-85-ii} If the family $\nu_{\theta}(x,dy)$, $\theta \in I$, is tight, then we can choose $\mc{D} := C_b^{\infty}(\mbb{R}^d)$ in Corollary~\ref{lk-8}\eqref{lk-8-ii}. Moreover, tightness implies \begin{equation*}
		\lim_{r \to 0} \sup_{\theta \in I} \sup_{|\xi| \leq r} |q_{\theta}(x,\xi)| =0,
		\end{equation*}
	cf.\  \cite[Lemma A.2]{hjb}. Conversely, the family $\nu_{\theta}(x,\cdot)$, $\theta \in I$, cannot be tight if the above limit does not equal zero. For instance, for $q_{\theta}(x,\xi) := 1-\cos(\theta \xi)$, $\theta \in I:=\mbb{N}$, tightness fails to  hold. See Proposition~\ref{lk-13} for an equivalent characterization of tightness in terms of the generator.
\end{enumerate} \end{bem}

\begin{proof}[Proof of Corollary~\ref{lk-8}]
	Set $Lf(x) := \sup_{\theta \in I}(-q_{\theta}(x,D)f)(x)$. Note that, by \eqref{pseudo2} and \eqref{def-eq11}, $Lf$ is well defined for any $f \in C_b^2(\mbb{R}^d)$.  Throughout this proof, $\chi \in C_c^{\infty}(\mbb{R}^d)$ is such that $\I_{B(0,1)} \leq \chi \leq \I_{B(0,2)}$, and  $\chi_r(x) := \chi(x/r)$ for $r>0$. \begin{enumerate}[wide, labelwidth=!, labelindent=0pt]
		\item For fixed $x \in \mbb{R}^d$ and $f \in C_b^{\infty}(\mbb{R}^d)$ define $f_n(y) := f(y) \chi_n(y-x)$. Since $f_n(x)=f(x)$ and $T_t(f) \leq T_t(f-f_n)+T_t(f_n)$, we have \begin{align*}
			\frac{T_t f(x)-f(x)}{t} - Lf(x) \leq \frac{T_t (f-f_n)(x)}{t} + \frac{T_t f_n(x)-f_n(x)}{t} - Lf(x).
		\end{align*}
		On the other hand, $T_t f_n \leq T_t (f_n-f) + T_tf$ gives \begin{align*}
			Lf(x) - \frac{T_t f(x)-f(x)}{t} \leq Lf(x) - \frac{T_t f_n(x)-f_n(x)}{t} + \frac{T_t (f_n-f)(x)}{t}.
		\end{align*}
		By definition, $f_n \in C_c^{\infty}(\mbb{R}^d)$, and so $A^{(p)} f_n = Lf_n$. Using $|f_n-f| \leq 2 \|f\|_{\infty} (1-\chi_n(\cdot-x))$, we get \begin{align*}
			\limsup_{t \to 0} \left| \frac{T_t f(x)-f(x)}{t} - Lf(x)\right| \leq 2\|f\|_{\infty} \limsup_{t \to 0} \frac{1+T_t(-\chi_n(\cdot-x))(x)}{t}  + |Lf_n(x)-Lf(x)|.
		\end{align*}
		As $\chi_n(\cdot-x) \in C_c^{\infty}(\mbb{R}^d)$, this gives \begin{equation*}
				\limsup_{t \to 0} \left| \frac{T_t f(x)-f(x)}{t} - Lf(x)\right| \leq 2\|f\|_{\infty} L(-\chi_n(\cdot-x))(x) + |Lf_n(x)-Lf(x)|.
		\end{equation*}
		It remains to show that the right-hand side converges to $0$ as $n \to \infty$. Since $\chi_n=0$ on $B(0,n)$, it follows from the definition of $L$ and the tightness of the L\'evy measures that \begin{equation*}
			L(-\chi_n(\cdot-x))(x) = \sup_{\theta \in I} \int_{y \neq 0} (1-\chi_n(y)) \, \nu_{\theta}(x,dy) \leq \sup_{\theta \in I} \int_{|y| \geq n} \nu_{\theta}(x,dy) \xrightarrow[]{n \to \infty} 0.
		\end{equation*}
		For the second term, we use the elementary inequality \begin{equation*}
			\sup_{i} a_i - \sup_i b_i \leq \sup_i (a_i-b_i),
		\end{equation*}
		the estimate \eqref{def-eq11} and the fact that $f_n = f$ on $B(x,n)$ to deduce that \begin{align*}
			Lf(x)-Lf_n(x) 
			&\leq 
			\sup_{\theta \in I} \left(-q_{\theta}(x,D)f(x)+q_{\theta}(x,D) f_n \right)(x) \\
			&\leq \left(\|f_n\|_{C_b^2(\mbb{R}^d)}+\|f\|_{C_b^2(\mbb{R}^d)}\right) \sup_{\theta \in I} \int_{|y| \geq n} \nu_{\theta}(x,dy).
		\end{align*}
		Interchanging the roles of $f$ and $f_n$, we get \begin{equation*}
			|Lf(x)-Lf_n(x)| \leq c \|f\|_{C_b^2(\mbb{R}^d)} \sup_{\theta \in I} \int_{|y| \geq n} \nu_{\theta}(x,dy) \xrightarrow[]{n \to \infty} 0.
		\end{equation*}
		\item Fix $x \in \mbb{R}^d$. From the proof of Theorem~\ref{lk-1}, cf.\ \eqref{lk-eq15}, we know that there exists a family $(A_{\theta})_{\theta \in \Theta}$ of linear functionals on $\mc{D}$ satisfying a positive maximum principle in $x \in \mbb{R}^d$ such that \begin{equation}
			A^{(p)} f(x) = \sup_{\theta \in \Theta} A_{\theta}f, \qquad f \in \mc{D}, \label{lk-eq31}
		\end{equation}
		and \begin{equation}
			A_{\theta}f = - q_{\theta}(x,D) f(x), \qquad f \in C_c^{\infty}(\mbb{R}^d). \label{lk-eq32}
		\end{equation}
		Let $u \in C_b^{\infty}(\mbb{R}^d)$ be such that $u=0$ on $B(x,2r)$ for some $r>0$. Since \begin{equation*}
			g(y) := \|u\|_{\infty} (1-\chi_r(y-x))-u(y) \geq 0 = g(x),
		\end{equation*}
		 it follows from the positive maximum principle that $A_{\theta} g \geq 0$, i.e.\ $A_{\theta} u \leq \|u\|_{\infty} A_{\theta}(1-\chi_r(\cdot-x))$. Replacing $u$ by $-u$, we find that \begin{equation}	
			|A_{\theta} u| \leq \|u\|_{\infty} |A_{\theta} (1-\chi_r(\cdot-x))|. \label{lk-eq33}
		\end{equation}
		Now let $f \in \mc{D} \subseteq C_b^{\infty}(\mbb{R}^d)$ and set $f_n(y) := f(y) \chi_n((y-x)/2n)$. Since $f-f_n=0$ on $B(x,2n)$, we get from \eqref{lk-eq32} and \eqref{lk-eq33} that \begin{align*}
			&|A_{\theta} (f) - q_{\theta}(x,D)f(x)| \\
			&\leq |A_{\theta}(f-f_n)| + |A_{\theta}(f_n) - q_{\theta}(x,D) f_n(x)| + |q_{\theta}(x,D)f_n(x)-q_{\theta}(x,D)f(x)|\\
			&\leq \|f-f_n\|_{\infty} |A_{\theta} (1-\chi_n(\cdot-x))|  + |q_{\theta}(x,D)f_n(x)-q_{\theta}(x,D)f(x)|.
		\end{align*}
		Using the representation of $q_{\theta}(x,D)$ as an integro--differential operator, cf.\ \eqref{pseudo2}, it can be easily verified that $q_{\theta}(x,D) f_n(x) \xrightarrow[]{n \to \infty} q_{\theta}(x,D) f(x)$. Moreover, as in \eqref{lk-8-i}, we find for fixed $\theta \in \Theta$ that $q_{\theta}(x,D)(1-\chi_n(\cdot-x))(x)$ converges to $0$ as $n \to \infty$. We conclude that $A_{\theta}f=q_{\theta}(x,D) f(x)$ for any $f \in C_b^{\infty}(\mbb{R}^d)$, and by \eqref{lk-eq31} this proves the assertion. \qedhere
\end{enumerate}	
\end{proof}

\begin{kor} \label{lk-125}
	Let $(T_t)_{t \geq 0}$ be a sublinear Markov semigroup with pointwise generator $(A^{(p)},\mc{D}(A^{(p)})$ satisfying $C_c^{\infty}(\mbb{R}^d) \subseteq \mc{D}(A^{(p)})$. Denote by $q_{\theta}(x,\cdot)$, $\theta \in I$, $x \in \mbb{R}^d$, the family of continuous negative definite functions associated with $(T_t)_{t \geq 0}$ via Corollary~\ref{lk-3}. If $(T_t)_{t \geq 0}$ is conservative, then $q_{\theta}(x,0)=0$ for all $\theta \in I$ and $x \in \mbb{R}^d$. 
\end{kor}

Note that $q_{\theta}(x,0)=0$ is equivalent to $c_{\theta}(x)=0$ in the representation \eqref{lk-eq55} of $A^{(p)}$ as an integro-differential operator. If $(T_t)_{t\geq0}$ is a \emph{linear} Markov semigroup, then the index set $I$ consists of a single element, i.e.\ the family $q(x,\cdot)$, $x \in \mbb{R}^d$, does not depend on an additional parameter $\theta$. Hence, Corollary~\ref{lk-125} shows that conservativeness of the semigroup implies $q(x,0)=0$ for all $x \in \mbb{R}^d$. This extends \cite[Lemma 5.1]{rs-cons}, see also \cite[Lemma 2.32]{ltp}, where the statement was shown for (linear) Feller semigroups under the additional assumption that $x \mapsto q(x,0)$ is continuous.

\begin{proof}[Proof of Corollary~\ref{lk-125}]
	Define a linear space by $\mc{D} := \{f+c \I_{\mbb{R}^d}; f \in C_c^{\infty}(\mbb{R}^d), c \in \mbb{R}\}$. From the conservativeness of $(T_t)_{t \geq 0}$, it follows that $T_t(f+c )=T_t(f)+c$ for all $f \in \mc{H}$ and $c \in \mbb{R}$, cf.\ \cite[Lemma 3.4]{julian}, and therefore $\mc{D} \subseteq \mc{D}(A^{(p)})$ and \begin{equation*}
		A^{(p)}(f+c \I_{\mbb{R}^d}) = A^{(p)}(f), \qquad f \in C_c^{\infty}(\mbb{R}^d), \,c \in \mbb{R}^d.
	\end{equation*}
	Applying Corollary~\ref{lk-8}\eqref{lk-8-ii}, we find for $f:=0$ that \begin{equation*}
		0 = A^{(p)}(c \I_{\mbb{R}^d})(x) = \sup_{\theta \in I} q_{\theta}(x,0), \qquad x \in \mbb{R}^d.
	\end{equation*}
	As $q_{\theta}(x,0) \geq 0$, this implies $q_{\theta}(x,0)=0$ for all $x \in \mbb{R}^d$ and $\theta \in I$.
\end{proof}

We close this section with an equivalent characterization of tightness of the family $\nu_{\theta}(x,\cdot)$, $\theta \in I$.

\begin{prop} \label{lk-13}
	Let $q_{\theta}(x,\cdot)$, $\theta \in I$, $x \in \mbb{R}^d$, be a family of continuous negative definite functions with $q_{\theta}(x,0)=0$. Denote by $(b_{\theta}(x),Q_{\theta}(x),\nu_{\theta}(x,\cdot))$ the associated family of triplets and set \begin{equation*}
		Lf(x) := \sup_{\theta \in I} (-q_{\theta}(x,D) f)(x), \qquad f \in C_b^{\infty}(\mbb{R}^d).
	\end{equation*}
	The following statements are equivalent for each $x \in \mbb{R}^d$. \begin{enumerate}
		\item\label{lk-13-i} The family $\nu_{\theta}(x,\cdot)$, $\theta \in I$, is tight.
		\item\label{lk-13-ii} For any $\eps>0$ there exists some $\varphi \in C_c^{\infty}(\mbb{R}^d)$, $0 \leq \varphi \leq 1$, such that $L(1-\varphi)(x) \leq \eps$ and $\varphi=1$ in a neighbourhood of $x$.
	\end{enumerate}
\end{prop}

\begin{proof}
	Fix $x \in \mbb{R}^d$ and $\eps>0$. If $\nu_{\theta}(x,\cdot)$, $\theta \in I$, is tight, then there exists $R>|x|$ such that \begin{equation*}
		\sup_{\theta \in I} \int_{|y|>R} \nu_{\theta}(x,dy) \leq \eps.
	\end{equation*}
	Choosing a function $\varphi \in C_c^{\infty}(\mbb{R}^d)$ such that $\varphi|_{B(0,R)}=1$ and $\varphi|_{B(0,2R)}=0$, we find that \begin{align*}
		L(1-\varphi)(x) = \sup_{\theta \in I} \int_{y \neq 0} (1-\varphi(y)) \, \nu_{\theta}(x,dy) \leq \sup_{\theta \in I} \int_{|y|>R} \nu_{\theta}(x,dy) \leq \eps.
	\end{align*}
	On the other hand, if $\varphi$ is a function as in \eqref{lk-13-ii}, then \begin{align*}
		\eps \geq
		 L(1-\varphi)(x)
		= \sup_{\theta \in I} \int_{y \neq 0} (1-\varphi(y)) \, \nu_{\theta}(x,dy)
		\geq \sup_{\theta \in I} \int_{|y|>R} \nu_{\theta}(x,dy)
	\end{align*}
	for $R \gg 1$ sufficiently large, and so the family of measures is tight.
\end{proof}

\section{Solutions to non-linear Cauchy problems associated with integro-differential operators} \label{visc}

In this section, we apply the Courr\`ege-von Waldenfels theorem to study solutions to the non-linear Cauchy problem \begin{equation}
	\frac{\partial}{\partial t} u(t,x) = \sup_{\theta \in I} (-q_{\theta}(x,D) u(t,\cdot))(x), \qquad u(0,x)=f(x) \label{visc-eq1}
\end{equation}
associated with a family of pseudo-differential operators $q_{\theta}(x,D)$, $\theta \in I$, cf.\ \eqref{pseudo1}. Because of the non-linearity, which is caused by the supremum, there exist, in general, no pointwise solutions to \eqref{visc-eq1}. We work with the weaker notion of viscosity solutions which was originally introduced by Crandall \& Lions \cite{lions} and Evans \cite{evans}. The following definition is taken from Hollender \cite{julian}. We refer the reader to \cite[Chapter 2]{julian} and \cite{barles07} for a discussion of equivalent definitions.

\begin{defn} \label{lk-10}
	Let $L: \mc{D}(L) \to \mbb{R}$ be an operator with domain $\mc{D}(L)$ containing the space of smooth functions with bounded derivatives $C_b^{\infty}(\mbb{R}^d)$. An upper semicontinuous function $u: [0,\infty) \times \mbb{R}^d \to \mbb{R}$ is a \emph{viscosity subsolution} to the equation \begin{equation*}
		\partial_t u(t,x)- L_x u(t,x)=0 
	\end{equation*}
	if the inequality $\partial_t \varphi(t,x)-L_x \varphi(t,x) \leq 0$ holds for any function $\varphi \in C_b^{\infty}([0,\infty) \times \mbb{R}^d)$ such that $u-\varphi$ has a global maximum in $(t,x) \in (0,\infty) \times \mbb{R}^d$ with $u(t,x) = \varphi(t,x)$. A mapping $u$ is a \emph{viscosity supersolution} if $-u$ is a viscosity subsolution. If $u$ is both a viscosity sub- and supersolution, then $u$ is called \emph{viscosity solution}.
\end{defn}

As usual, we write $L_x$ to indicate that $L$ acts with respect to the space variable $x$. In order to construct a viscosity solution to \eqref{visc-eq1}, we use the following fundamental theorem by Hollender \cite{julian} which associates with a sublinear semigroup $(T_t)_{t \geq 0}$ an evolution equation.

\begin{prop}[{\cite[Proposition 4.10]{julian}}]  \label{visc-3} 
	Let $(T_t)_{t \geq 0}$ be a sublinear Markov semigroup on $\mc{H}$. Assume that the domain of the pointwise infinitesimal generator $A^{(p)}$ contains the smooth functions with bounded derivatives. If $f \in \mc{H}$ is such that $(t,x) \mapsto u(t,x) = T_t f(x)$ is continuous, then $u$ is a viscosity solution to \begin{equation*}
		\frac{\partial}{\partial t} u(t,x) = A^{(p)}_x u(t,x), \qquad u(0,x)=f(x).
	\end{equation*}
\end{prop}

Combining Proposition~\ref{visc-3} with the Courr\`ege--Waldenfels theorem, we obtain the following corollary.

\begin{kor} \label{lk-11}
	Let $(T_t)_{t \geq 0}$ be a sublinear Markov semigroup on $\mc{H}$ with pointwise infinitesimal generator $(A^{(p)},\mc{D}(A^{(p)}))$ such that $C_b^{\infty}(\mbb{R}^d) \subseteq \mc{D}(A^{(p)})$. If $f \in \mc{H}$ is such that $(t,x) \mapsto u(t,x) := T_t f(x)$ is continuous, then $u$ is a viscosity solution to  \begin{equation}
		\frac{\partial}{\partial t} u(t,x) = \sup_{\theta \in I} (-q_{\theta}(x,D) u(t,\cdot))(x), \qquad u(0,x)=f(x) \label{lk-eq41}
	\end{equation}
	where $(q_{\theta}(x,\cdot))$ is the family of continuous negative definite functions from Corollary~\ref{lk-3}.
\end{kor}

\begin{proof}
	By Propsition~\ref{visc-3}, $u(t,x)=T_t f(x)$ is a viscosity solution to \begin{equation*}
		\frac{\partial}{\partial t} u(t,x) = A_x^{(p)} u(t,x), \qquad u(0,x)=f(x).
	\end{equation*}
	Applying Corollary~\ref{lk-8}\eqref{lk-8-ii} and using the very definition of the notion of viscosity solutions, cf.\ Definition~\ref{lk-10}, we get immediately the assertion.
\end{proof}

\begin{bem} \label{lk-12} \begin{enumerate}[wide, labelwidth=!, labelindent=0pt]
	\item\label{lk-12-i} If the associated family of L\'evy measures $\nu_{\theta}(x,\cdot)$, $\theta \in I$, is tight for each $x \in \mbb{R}^d$, then the assumption $C_b^{\infty}(\mbb{R}^d) \subseteq \mc{D}(A^{(p)})$ can be relaxed to $C_c^{\infty}(\mbb{R}^d) \subseteq \mc{D}(A^{(p)})$, see Corollary~\ref{lk-8}\eqref{lk-8-i}.
	\item Corollary~\ref{lk-11} requires that $(t,x)\mapsto T_t f(x)$ is continuous. Typically, continuity with respect to $t$ is much easier to verify than continuity with respect to $x$. For a large class of sublinear Markov semigroups, it is shown in \cite[Theorem 5.3]{hjb} that $t \mapsto T_t f(x)$ is continuous uniformly for $x$ in a compact set. There seems to be no general result which allows us to deduce continuity with respect to the space variable $x$, see also the discussion in \cite[Remark 4.43]{julian}. Translation invariant semigroups $(T_t)_{t \geq 0}$ are one of the few exceptions where continuity with respect to $x$ is easy to obtain; for instance, it is not difficult to see that $T_t f \in C_b^{\uc}(\mbb{R}^d)$ for any $f \in C_b^{\uc}(\mbb{R}^d)$.
	\item For recent existence results for the non-linear Cauchy problem \eqref{lk-eq41} see \cite{nendel19,denk} (via the dynamic programming principle) and \cite{julian,hjb} (via processes under uncertainty) and the references therein.
\end{enumerate}
\end{bem}

\section{L\'evy processes on sublinear expectation spaces} \label{levy}

It is well known that there is a one-to-one correspondence between (classical) L\'evy processes and (linear) translation invariant Markov semigroups, see e.g.\ \cite[Section 2.1]{ltp}. Recently, Denk et.\ al \cite{denk} obtained a similar result in the framework of non-linear semigroups and processes on non-linear expectation spaces. Let us first give the definitions, see also Section~\ref{def}.

\begin{defn} \label{levy-1}
	We call a family of sublinear operators $T_t: \mc{H} \to \mc{H}$, $t \geq 0$, a \emph{sublinear Markov convolution semigroup (on $\mc{H}$)} if \begin{enumerate}
		\item $(T_t)_{t \geq 0}$ is a sublinear conservative Markov semigroup on $\mc{H}$,
		\item $T_t$ is translation invariant for each $t \geq 0$, i.e.\ $f(x+\cdot) \in \mc{H}$ and $(T_t f)(x) = (T_t f(x+\cdot))(0)$ for all $x \in \mbb{R}^d$ and $f \in \mc{H}$,
		\item $(T_t)_{t \geq 0}$ is strongly continuous at $t=0$, i.e.\ $\|T_t f-f\|_{\infty} \to 0$ as $t \to 0$ for all $f \in \mc{H}$.
	\end{enumerate}
\end{defn}

\begin{defn} \label{levy-3}
	Let $(\Omega,\mc{H},\mc{E})$ be a sublinear expectation space, cf.\ Section~\ref{def}. A family $X_t: \Omega \to \mbb{R}^d$, $t \geq 0$, is a \emph{L\'evy process for sublinear expectations} (or \emph{sublinear L\'evy process}) if \begin{enumerate}
		\item $X_t$ is adapted forall $t \geq 0$, i.e.\ $f(X_t) \in \mc{H}$ for all $f \in C_b^{\uc}(\mbb{R}^d)$,
		\item $X_0 \stackrel{d}{=} 0$ in distribution,
		\item $X_{t_1}-X_{t_0},\ldots,X_{t_n}-X_{t_{n-1}}$ are independent for any $0=t_0 < \ldots < t_n$, $n \in \mbb{N}$ (\emph{independent increments}),
		\item $X_t-X_s \stackrel{d}{=} X_{t-s}$ for all $s \leq t$ (\emph{stationary increments}),
		\item $X_t \stackrel{d}{\to} 0$ as $t \downarrow 0$, i.e.\ $\mc{E}f(X_t) \to f(0)$ for any $f \in C_b^{\uc}(\mbb{R}^d)$.
	\end{enumerate}
\end{defn}

If $(X_t)_{t \geq 0}$ is a L\'evy process for sublinear expectations, then $T_t f(x) := \mc{E}f(x+X_t)$ defines a sublinear Markov convolution semigroup on $C_b^{\uc}(\mbb{R}^d)$, see \cite[(Proof of) Theorem 2.3]{denk} and also \cite[Remark 4-38]{julian}. Conversely, sublinear Markov convolution semigroups can be used to construct L\'evy processes for sublinear expectations, cf.\  \cite{denk}. \par 

Classical L\'evy processes can be uniquely characterized (in distribution) by their L\'evy triplet $(b,Q,\nu)$, and it is natural to ask whether there is an analogous result for sublinear L\'evy processes. In this section, we show that the answer is positive for ``nice'' sublinear L\'evy processes, see Theorem~\ref{levy-9} below. Instead of a single L\'evy triplet, we will be dealing with a family  of triplets $(b_{\theta},Q_{\theta},\nu_{\theta})$, $\theta \in I$, which is obtained from the Courr\`ege-von Waldenfels theorem.

\begin{defn} \label{levy-4}
	Let $(X_t)_{t \geq 0}$ be a L\'evy process for sublinear expectations with semigroup $(T_t)_{t \geq 0}$. If the associated pointwise infinitesimal generator $(A^{(p)},\mc{D}(A^{(p)})$ satisfies $C_c^{\infty}(\mbb{R}^d) \subseteq \mc{D}(A^{(p)})$, then we say that (the generator of) $(X_t)_{t \geq 0}$ \emph{has a rich domain}. By the Courr\`ege--Waldenfels theorem, Corollary~\ref{lk-9}, and Corollary~\ref{lk-125}, we can associate a family $(b_{\theta},Q_{\theta},\nu_{\theta})$, $\theta \in I$, with any such process. We call this family  \emph{characteristics} of $(X_t)_{t \geq 0}$.
\end{defn}

Sublinear L\'evy processes can be interpreted as stochastic processes under uncertainty, cf.\ Hollender \cite{julian}, and therefore $(b_{\theta},Q_{\theta},\nu_{\theta})$, $\theta \in I$, are sometimes called uncertainty coefficients.  The following theorem is the main result in this section.

\begin{thm} \label{levy-9}
	Let $(b_{\theta},Q_{\theta},\nu_{\theta})$, $\theta \in I$, be a uniformly bounded family of triplets, i.e.\ \begin{equation*}
		\sup_{\theta \in I} \left( |b_{\theta}| + |Q_{\theta}| + \int_{y \neq 0} \min\{1,|y|^2\} \, \nu_{\theta}(dy) \right) < \infty.
	\end{equation*} \begin{enumerate}
		\item\label{levy-9-i} There exists a L\'evy process for sublinear expectations with characteristics $(b_{\theta},Q_{\theta},\nu_{\theta})_{\theta \in I}$. 
		\item\label{levy-9-ii} If $\nu_{\theta}$, $\theta \in I$, is tight, then there exists a unique sublinear L\'evy process with characteristics $(b_{\theta},Q_{\theta},\nu_{\theta})_{\theta \in I}$, i.e.\ any two sublinear L\'evy processes with the given characteristics have the same finite-dimensional distributions. 
	\end{enumerate}
\end{thm}

There are several possibilities to construct a sublinear L\'evy process with a given characteristics, e.g.\ via the dynamic programming principle \cite{denk,nendel19} or as process under uncertainty \cite{julian,hjb}. Theorem~\ref{levy-9} tells us, in particular, that for nice triplets (i.e.\ if tightness holds) these constructions yield the same process, i.e.\ the constructed processes have the same finite-dimensional distributions and the same semigroup. \par \medskip

For the proof of Theorem~\ref{levy-9} the following result plays a crucial role. It gives a sufficient condition ensuring that a Markov convolution semigroup is uniquely determined by its pointwise generator restricted to $C_c^{\infty}(\mbb{R}^d)$.  The recent paper \cite{denk19} studies in a more general framework under which conditions a sublinear semigroup is uniquely determined by its generator.

\begin{prop} \label{levy-5}
	Let $(P_t)_{t \geq 0}$ and $(T_t)_{t \geq 0}$ be sublinear Markov convolution semigroups on $\mc{H} \subseteq C_b(\mbb{R}^d)$. Assume that the domains of the pointwise generators $(A^{(p)},\mc{D}(A^{(p)}))$ and $(L^{(p)},\mc{D}(L^{(p)}))$ contain $C_c^{\infty}(\mbb{R}^d)$ and \begin{equation*}
		\forall f \in C_c^{\infty}(\mbb{R}^d) \::\: A^{(p)} f = L^{(p)} f.
	\end{equation*}
	Denote by $(b_{\theta},Q_{\theta},\nu_{\theta})$, $\theta \in I$, the associated family of L\'evy triplets, cf.\ Corollary~\ref{lk-9}. If the family of measures $\nu_{\theta}$, $\theta \in I$, is tight, i.e.\ \begin{equation}
		\lim_{R \to \infty} \sup_{\theta \in I} \int_{|y|>R} \nu_{\theta}(dy)=0, \label{levy-11}
	\end{equation}
	then $P_t f = T_t f$ for all $t \geq 0$ and $f \in \mc{H}$.
\end{prop}

\begin{proof}
	Let $f \in \mc{H}$. First we show that the mappings $(t,x) \mapsto T_t f(x)$ and $(t,x) \mapsto P_t f(x)$ are continuous; clearly, it suffices to consider one of the semigroups. For fixed $s \leq t$ the subadditivity and monotonicity of the operators yield \begin{equation*}
		T_t - T_s f = T_s T_{t-s} f - T_s f \leq T_s(T_{t-s}f-f) \leq \|T_{t-s} f-f\|_{\infty}.
	\end{equation*}
	Interchanging the roles of $s$ and $t$ we get \begin{equation*}
		\|T_tf-T_s f\|_{\infty} \leq \|T_{|t-s|} f-f\|_{\infty}, \qquad s,t \geq 0.
	\end{equation*}
	Hence, \begin{align*}
		|T_{t} f(x)-T_s f(y)|
		\leq |T_t f(x)-T_t f(y)| + |T_t f(y)-T_s f(y)| 
		\leq |T_t f(x)-T_t f(y)| + \|T_{|t-s|} f-f\|_{\infty}.
	\end{align*}
	Since $T_t f \in \mc{H} \subseteq C_b(\mbb{R}^d)$ is continuous and the semigroup is strongly continuous at $r=0$, we find that the right-hand side converges to $0$ if we let $s \to t$ and $y \to x$. Hence, $(t,x) \mapsto T_t f(x)$ is continuous.  
	By Corollary~\ref{lk-11} and Remark~\ref{lk-12}\eqref{lk-12-i}, $(t,x) \mapsto T_t f(x)$ and $(t,x) \mapsto P_t f(x)$ are viscosity solutions to the evolution equation \begin{equation*}
		\frac{\partial}{\partial t} u(t,x) = \sup_{\theta \in I} (-\psi_{\theta}(D) u(t,\cdot))(x), \qquad u(0,x)=f(x).
	\end{equation*}
	Because of the tightness condition \eqref{levy-11}, the Cauchy problem has a unique viscosity solution, cf.\ \cite[Corollary 2.34]{julian}, and so $P_t f = T_t f$ for all $t \geq 0$. 
\end{proof}

After these preparations, we are ready to prove Theorem~\ref{levy-9}.

\begin{proof}[Proof of Theorem~\ref{levy-9}]
	It follows from \cite[Remark 4.38]{julian} and \cite[Proposition 4.1, Corollary 4.2]{hjb} that there exist a measurable space $(\Omega,\mc{A})$, a family of probability measures $(\mbb{P})_{\mbb{P} \in \mathfrak{P}}$ and a stochastic process $(X_t)_{t \geq 0}$ with the following properties: \begin{itemize}
		\item $X_0 =0$ and $(X_t)_{t \geq 0}$ has stationary and independent increments,
		\item $T_t f(x) := \sup_{\mbb{P} \in \mathfrak{P}} \mbb{E}_{\mbb{P}} f(x+X_t)$ defines a sublinear Markov semigroup on $\mc{H} := C_b^{\uc}(\mbb{R}^d)$; here $\mbb{E}_{\mbb{P}}$ denote the expectation with respect to $\mbb{P}$.
		\item $C_b^2(\mbb{R}^d)$ is contained in the domain of the pointwise generator $A^{(p)}$,
		\item $A^{(p)} f(x) = \sup_{\theta \in I} (-\psi_{\theta}(D) f)(x)$ for all $f \in C_b^2(\mbb{R}^d)$.
	\end{itemize}
	Moreover, it follows from \cite[Theorem 5.3]{hjb} that $(T_t)_{t \geq 0}$ is strongly continuous on $C_b^{\uc}(\mbb{R}^d)$ and that $T_t f(0) \to f(0)$ for all $f \in C_b(\mbb{R}^d)$.  This proves \eqref{levy-9-i}.  \par
	Now assume that the family $\nu_{\theta}$, $\theta \in I$, is tight, and let $(X_t)_{t \geq 0}$ and $(Y_t)_{t \geq 0}$ be two sublinear L\'evy processes with rich domains and characteristics $(b_{\theta},Q_{\theta},\nu_{\theta})_{\theta \in I}$. The associated semigroups $(P_t)_{t \geq 0}$ and $(T_t)_{t \geq 0}$ are sublinear Markov convolution semigroups on $\mc{C}_b^{\uc}(\mbb{R}^d)$, cf.\ \cite[(Proof of) Theorem 2.3]{denk}. By assumption, their pointwise generators coincide on $C_c^{\infty}(\mbb{R}^d)$. Applying Corollary~\ref{lk-8} and Proposition~\ref{levy-5}, we find that $P_t f = T_t f$ for any $f \in \mc{C}_b^{\uc}(\mbb{R}^d)$ and $t \geq 0$. It is easily seen from the independence of the increments that the finite dimensional distributions of a sublinear L\'evy process are uniquely determined by its semigroup, and so the assertion follows.
\end{proof}

\end{document}